\documentclass[11pt,leqno]{amsart}
\usepackage{amsmath,amssymb,amsthm}
\usepackage{hyperref}
\newtheorem{theorem}{Theorem}[section]
\newtheorem{lemma}[theorem]{Lemma}
\newtheorem{claim}[theorem]{Claim}
\newtheorem{proposition}[theorem]{Proposition}
\newtheorem{corollary}[theorem]{Corollary}
\theoremstyle{definition}
\newtheorem{definition}[theorem]{Definition}
\newtheorem{example}[theorem]{Example}

\newtheorem{question}[theorem]{Question}
\theoremstyle{remark}
\newtheorem{remark}[theorem]{Remark}

\numberwithin{equation}{section}
\def\fnote#1{\footnote}

\def\real{{\mathbb R}}

\def\ignora#1{}
\def\n3#1{\left\vert  \! \left\vert \! \left\vert \, #1 \, \right\vert \!
  \right\vert \! \right\vert }

\newcommand{\iten}{\ensuremath{\widehat{\otimes}_\varepsilon}}
\newcommand{\pten}{\ensuremath{\widehat{\otimes}_\pi}}

\begin{document}

\title[UASQ and applications to spaces of Lipschitz functions]{Unconditional  almost squareness and applications to spaces of Lipschitz functions}
\author{Luis Garc\'ia-Lirola}\thanks{First author was partially supported by the grants MINECO/FEDER MTM2014-57838-C2-1-P and Fundaci\'on S\'eneca CARM 19368/PI/14 and by the Region de Franche-Comt\'e.}
\address{Departamento de Matem\'aticas, Facultad de Matem\'aticas, Universidad de Murcia,
30100-Murcia, Spain}\email{luiscarlos.garcia@um.es}

\author{ Abraham Rueda Zoca}\thanks{The second author was partially supported by Junta de Andaluc\'ia Grants FQM-0199.}
\address{Universidad de Granada, Facultad de Ciencias.
Departamento de An\'{a}lisis Matem\'{a}tico, 18071-Granada, Spain} \email{ arz0001@correo.ugr.es}



\begin{abstract}
We introduce an unconditional concept of almost squareness in order to provide a partial negative answer to the problem of existence of any dual almost square Banach space. We also take advantage of this notion to provide some criterion of non-duality of some subspaces of scalar as well as vector valued Lipschitz functions.
\end{abstract}

\maketitle

\section{Introduction}

The space $Lip(M)$ of Lipschitz functions vanishing at $0$ over a pointed metric space $M$ as well as their canonical preduals $\mathcal F(M)$ have received much attention since Godefroy and Kalton paper \cite{goka}. Indeed, the topological as well as the geometrical structure of such spaces have been deeply analysed since then (see \cite{cdw, god, hlp, ikw, kal, lp} and references therein for background). One of the most traditional problems is trying to determine when $\mathcal F(M)$ is itself a dual Banach space and, in that case, trying to identify some preduals as a subspace of $Lip(M)$. In this line, two candidates for preduals of $\mathcal F(M)$ appeared in the case of being $M$ compact or proper (i.e. closed balls are compact). On the one hand, for the compact case, such a natural candidate is $lip(M)$ (see \cite{wea}), the space of those Lipschitz functions which are uniformly locally flat (see formal definition below). On the other hand, in proper case, the role of a predual of Lipschitz-free Banach spaces is played by $S(M)$ (see \cite{dal2}), which is a subspace of $lip(M)$ whose elements have an additional behaviour of flatness at infinity (again, see formal definition below). It is known that $lip(M)$ (respectively $S(M)$) is an $M$-embedded Banach space whenever $M$ is compact \cite{kal} (respectively $M$ is proper \cite{dal2}) and, consequently, these spaces can not be dual Banach spaces whenever they are infinite dimensional. So, it is natural wondering when previous Banach spaces can be dual ones in a more general setting. For this, it would be useful to find a geometrical property of Banach spaces which is not compatible with being a dual Banach space. In this line, it has been recently introduced the concept of almost squareness. According to \cite{all}, a Banach space $X$ is said to be \textit{almost square} (ASQ) if for every $x_1,\ldots, x_k\in S_X$ and $\varepsilon>0$ there exists $y\in S_X$ such that
$$\Vert x_i\pm y\Vert\leq 1+\varepsilon\ \forall i\in \{1,\ldots, k\}.$$
Roughly speaking, we can say that ASQ Banach spaces have a strong $c_0$ behaviour from a geometrical point of view. This $c_0$ behaviour is encoded by the fact that a Banach space admits an equivalent renorming to be ASQ if, and only if, the space contains an isomorphic copy of $c_0$ \cite{blr2}. In this setting, it could be natural wondering whether this kind of spaces can be dual ones. Even though there are quite a lot examples of ASQ Banach spaces which are not dual (e.g. non-reflexive $M$-embedded Banach space \cite{all}), it is actually an open problem if there exists any dual and ASQ Banach space (posed in \cite{all} as well as in \cite{blr2}).

So, the aim of this note is to find a generalisation of the concept of almost squareness and to apply it to give some criteria on non-duality of $lip(M)$ and $S(M)$ as well as their vector-valued versions. Indeed, in Section \ref{sectionuasq} we will introduce the concept of \textit{unconditionally almost square Banach spaces} (see Definition \ref{definitiouasq}) and, after providing some sufficient conditions on unconditional almost squareness, we will prove the two main results of the Section. On the one hand, we prove in Theorem \ref{uncondASQnodual} that there is not any dual unconditionally almost square Banach space, which partially solves by the negative to the problem of whether there is any dual ASQ space. On the other hand, we prove a stability result for unconditional almost squareness in Proposition \ref{estabiluasqespaope} which, though the proof follows the ideas of \cite[Theorem 2.6]{llr}, it will be applied to analyse unconditional almost squareness in vector-valued versions of $lip(M)$ and $S(M)$ in Section \ref{secvectorvalued}. After that, Sections \ref{seccasonodiscre} and \ref{secunifdisc} are devoted to analyse unconditional almost squareness in $lip(M)$ and $S(M)$, proving for instance in Theorem \ref{discnouniuasq} that previous spaces are unconditionally almost square whenever $M$ is a locally compact and totally disconnected metric space which is not uniformly discrete. Finally, in Section \ref{secvectorvalued} we will get some result on unconditional almost squareness in vector-valued cases. For that, we will introduce the vector-valued version of $lip(M)$ and $S(M)$ ($lip(M,X)$ and $S(M,X)$ respectively) and we will prove in Theorem \ref{scomp} that, given $M$ a proper metric space, $S(M,X)$ is linearly isometrically isomorphic to the space of compact operators from $X^*$ to $S(M)$  which are $w^*-w$ continuous, which extends the results of compact case done in \cite{jsv}. As a consequence, we will get in Theorem \ref{UASQsmvectorvaluado} that $S(M,X)$ is unconditionally almost square whenever $S(M)$ is infinite-dimensional and, consequently, $S(M,X)$ can not be a dual Banach space.\\

\textbf{Notation}. Given $M$ a metric space, $B(x,r)$ (respectively $\overline{B}(x,r)$) denotes the open (respectively closed) ball in $M$ centered at $x\in M$ with radius $r$. In addition, $M'$ will denote the set of cluster points in $M$. Throughout the paper we will only consider real Banach spaces. Given a Banach space $X$, we will denote by $B_X$ (respectively $S_X$) the closed unit ball (respectively the unit sphere) of $X$. We will also denote by $X^*$ the topological dual of $X$. We will denote $X\pten Y$ (resp. $X\iten Y)$ the projective (resp. injective) tensor product of Banach spaces. For a detailed treatment and applications of tensor products, we refer the reader to \cite{rya}. In addition, $L(X,Y)$ (respectively $K(X,Y)$) will denote the space of continuous (respectively compact) operators from $X$ to $Y$. Moreover, given topologies $\tau_1$ on $X$ and $\tau_2$ on $Y$, we will denote by $L_{\tau_1,\tau_2}(X,Y)$ and $K_{\tau_1,\tau_2}(X, Y)$ the respective subspaces of $\tau_1$-$\tau_2$ continuous operators.   

Given $M$ a metric space with a designated origin $0$ and a Banach space $X$, we will denote by $Lip(M,X)$ the Banach space of all $X$-valued Lipschitz functions on $M$ which vanish at $0$ under the standard Lipschitz norm 
$$\Vert f\Vert:=\sup\left\{\left. \frac{\Vert f(x)-f(y)\Vert}{d(x,y)}\ \right/\ x,y\in M, x\neq y \right\} .$$
First of all, note that we can consider every point of $M$ as an origin with no loss of generality, because the resulting Banach spaces turn out to be isometrically isomorphic (see, e.g. \cite{blr}). Moreover, it is known that $Lip(M,X^*)$ is a dual Banach space, and it is defined in \cite{blr} a canonical predual given by
$$\mathcal F(M,X):=\overline{span}\{\delta_{m,x}: m\in M, x\in X\}\subseteq Lip(M,X^*)^*,$$
where $\delta_{m,x}(f):=f(m)(x)$ for every $m\in M, x\in X$ and $f\in Lip(M,X^*)$. Furthermore, it is known that, for every metric space $M$ and every Banach space $X$, $Lip(M,X)=L(\mathcal F(M),X)$ (e.g.\cite{jsv}) and that $\mathcal F(M,X)=\mathcal F(M)\pten X$ (see \cite[Proposition 1.1]{blr}).

We will consider the following spaces of vector-valued Lipschitz functions.  
\begin{align*}
lip(M,X)&:=\left\{f\in Lip(M,X): \exists\ \lim\limits_{\varepsilon\rightarrow 0} \sup\limits_{0<d(x,y)<\varepsilon} \frac{\Vert f(x)-f(y)\Vert}{d(x,y)}=0\right\},\\
S(M,X)&:=\left\{ f\in lip(M,X) : \exists\ \lim\limits_{r\rightarrow \infty} \sup_{\stackrel{x \text{ or } y\notin B(0,r)}{x\neq y}} \frac{\Vert f(x)-f(y)\Vert}{d(x,y)}=0\right\}.
\end{align*}
We will avoid the reference to the Banach space when it is $\mathbb R$ in the above definitions. We finally refer to \cite{hww} for background about theory of $M$-embedded Banach spaces.

\section{An uniform sense of almost square Banach space}\label{sectionuasq}

\bigskip

We shall begin by providing a stronger notion of almost squareness in Banach spaces. 

\begin{definition}\label{definitiouasq}
Let $X$ be a Banach space. We will say that $X$ is \textit{unconditionally almost square (UASQ)} if, for each $\varepsilon>0$, there exists a subset $\{x_\gamma\}_{\gamma\in \Gamma}\subseteq S_X$ (depending on $\varepsilon$) such that

\begin{enumerate}
\item\label{defi211} For each $\{y_1,\ldots, y_k\}\subseteq S_X$ and $\delta>0$ the set
$$\{\gamma\in\Gamma :  \Vert y_i\pm x_\gamma\Vert\leq 1+\delta\ \forall i\in\{1,\ldots, k\}\}$$
is non-empty.
\item\label{defi212} For every $F$ finite subset of $\Gamma$ and every choice of signs $\xi_\gamma\in \{-1,1\}$, $\gamma\in F$, it follows $\Vert\sum_{\gamma\in F} \xi_\gamma x_\gamma\Vert\leq 1+\varepsilon$.
\end{enumerate}
\end{definition}
First of all, we shall exhibit some examples of such spaces:
\begin{example}
\begin{enumerate}
\item The space $c_0(\Gamma)$ is UASQ, where $\{e_\gamma\}_{\gamma\in \Gamma}$ works for every $\varepsilon>0$.
\item Consider $\Gamma$ an uncountable set and consider $\ell_\infty^c(\Gamma):=\{x\in\ell_\infty(\Gamma): supp(x)\mbox{ is countable}\}$. Then $\ell_\infty^c(\Gamma)$ is UASQ, where the set $\{e_\gamma\}_{\gamma\in\Gamma}$ works for every $\varepsilon>0$.
\item Given $\Gamma$ an infinite set and $\mathcal U$ a free ultrafilter over $\Gamma$, the space $X:=\{x\in\ell_\infty(\Gamma):\lim_\mathcal U (x)=0\}$ is UASQ, where the set $\{e_\gamma:\gamma\in \Gamma\}$ works for every $\varepsilon>0$.
\end{enumerate}
\end{example}

Now let us exhibit a general result which will give us a wide class of examples of UASQ spaces.

\begin{proposition}\label{condisufiuasq}
Let $X$ be a Banach space. Assume that there exists a sequence $\{x_n\}$ in $S_X$ such that, for each $\varepsilon>0$ and every $y\in S_X$, there exists $m\in\mathbb N$ such that $n\geq m$ implies $\Vert y\pm x_n\Vert\leq 1+\varepsilon$. Then $X$ is an UASQ Banach space.
\end{proposition}

\begin{proof}
Pick $\{x_n\}$ to be the sequence of the assumptions. It is clear that the condition is equivalent to the fact that given $\{y_1,\ldots, y_k\}\subseteq S_X$ there exists $m\in\mathbb N$ such that $n\geq m$ implies $\Vert y_i\pm x_n\Vert\leq 1+\varepsilon$ for each $i\in\{1,\ldots, k\}$. Moreover, unconditional almost squareness will be proved just taking suitable subsequences of $\{x_n\}$ for every  $\varepsilon$. We will focus on verifying condition (\ref{defi212}) in Definition \ref{definitiouasq} because condition (\ref{defi211}) will be satisfied whenever we are dealing with a subsequence of $\{x_n\}$.

Fix $\varepsilon>0$. Pick a sequence of positive numbers $\{\varepsilon_n\}$ such that $\prod_{n=1}^\infty (1+\varepsilon_n)<1+\varepsilon$ and $\prod_{n=1}^\infty (1-\varepsilon_n)>1-\varepsilon$. Consider $\sigma(1):=1$. From assumptions there exists $\sigma(2)\in \mathbb N\setminus\{1\}$ such that $n\geq \sigma(2)$ implies $\Vert \xi_1x_{\sigma(1)}+\xi_2x_n\Vert<1+\varepsilon_1$ for every $\xi_1,\xi_2\in \{-1,1\}$.  Moreover, by \cite[Lemma 2.2]{all} we can ensure that 
$$1-\varepsilon_1<\Vert \xi_1x_{\sigma(1)}+\xi_2x_{\sigma(2)}\Vert<1+\varepsilon_1\ \ \mbox{for all}\ \xi_1,\xi_2\in \{-1,1\}.$$
Again by assumptions we can find $\sigma(3)\in\mathbb N$ such that $\sigma(3)>\sigma(2)$ and that $n\geq \sigma(3)$ implies $\left\Vert \frac{\xi_1x_{\sigma(1)}+\xi_2x_{\sigma(2)}}{\Vert \xi_1x_{\sigma(1)}+\xi_2x_{\sigma(2)}\Vert}\pm\xi_3 x_n\right\Vert<1+\varepsilon_2$ for all $\xi_1,\xi_2,\xi_3\in \{-1,1\}$. By \cite[Lemma 2.2]{all}, we get that
$$(1-\varepsilon_1)(1-\varepsilon_2)<\left\Vert \sum_{i=1}^3\xi_ix_{\sigma(i)}\right\Vert<(1+\varepsilon_1)(1+\varepsilon_2).$$
By this procedure we can get a subsequence $\{x_{\sigma(n)}\}$ of $\{x_n\}$ such that, given $n\in \mathbb N$ and $\xi_1,\ldots, \xi_n\in \{-1,1\}$, we get
$$1-\varepsilon\leq\prod_{i=1}^{n-1}(1-\varepsilon_i)<\left\Vert \sum_{i=1}^n \xi_i x_{\sigma(i)}\right\Vert<\prod_{i=1}^{n-1}(1+\varepsilon_i)\leq 1+\varepsilon,$$
which proves condition (\ref{defi212}) in definition of UASQ. Moreover, (\ref{defi211}) is clear.\end{proof}

From Proposition \ref{condisufiuasq} we conclude the following Corollary.

\begin{corollary}\label{uasqequiasqsep}
Let $X$ be a separable Banach space. If $X$ is ASQ, then $X$ is UASQ.
\end{corollary}

\begin{proof}
Pick $\{y_n\}$ to be a dense sequence in $S_X$. For each $n\in\mathbb N$ pick, from ASQ condition, an element $x_n\in S_X$ such that
$$\Vert y_i\pm x_n\Vert\leq 1+\frac{1}{n}\ \  \forall i\in \{1,\ldots, n\}.$$
We will prove that $\{x_n\}$ verifies the conditions on Proposition \ref{condisufiuasq}. To this aim pick $\varepsilon>0$ and $z\in S_X$. Now we can find a natural number $k$ such that $\Vert z-y_{k}\Vert<\frac{\varepsilon}{2}$. Pick $m$ large enough to ensure $m>k$ and $\frac{1}{m}<\frac{\varepsilon}{2}$. Given $n\geq m$ one has
$$\Vert z\pm x_n\Vert\leq \Vert z-y_{k}\Vert +\Vert y_{k}\pm x_n\Vert\leq 1+ \frac{\varepsilon}{2}+\frac{1}{n}<1+\varepsilon,$$
which finishes the proof.\end{proof}

Now we will prove the main result of the section, which justifies the definition of UASQ spaces.

\begin{theorem}\label{uncondASQnodual}
Let $X$ be a Banach space. Then $X^*$ can not be UASQ.
\end{theorem}

\begin{proof}
Assume by contradiction that $X^*$ is UASQ. Note that $X^*$ can be seen as a subspace of $\ell_\infty(I)$, for certain infinite set $I$. Indeed, given $I=S_X$, we have that
\[ \begin{array}{cccc}
\Phi\colon & X^* & \longrightarrow & \ell_\infty(S_X)\\
& x^* & \longmapsto & \Phi(x^*)(x)=x^*(x)
\end{array}\]
defines a linear isometry which is $w^*-w^*$ continuous (indeed, it is not difficult to prove that $\Phi$ is the adjoint of the operator $\Psi:\ell_1(S_X)\longrightarrow X$ defined by $\Psi(f):=\sum_{x\in S_X} f(x)x$ for every $f\in\ell_1(S_X)$).

From here, in order to save notation, we will just simply assume that $X^*$ is a subspace of $\ell_\infty(I)$ whose balls are weak-star closed in $\ell_\infty(I)$.

Pick $\varepsilon>0$ and  consider $\{x_\gamma\}_{\gamma\in \Gamma}$ the set of the definition of UASQ Banach space. Pick $k\in I$ arbitrary and choose, for each $\gamma\in \Gamma$, an element $\xi_\gamma^k\in \{-1,1\}$ such that $\xi_\gamma^k x_\gamma(k)=\vert x_\gamma(k)\vert$. Given $F$ a finite subset of $\Gamma$ one has 
\begin{equation}\label{sumabsolutafami}
1+\varepsilon\geq \left\vert \sum_{\gamma\in F} \xi_\gamma^k x_\gamma(k)\right\vert=\sum_{\gamma\in F} \vert x_\gamma(k)\vert\Rightarrow \exists\ \sum_{\gamma\in \Gamma} \vert x_\gamma(k)\vert\leq 1+\varepsilon\ \forall k\in I.
\end{equation}
Now, for each $k\in I$, we conclude that $\{x_\gamma(k)\}$ is an absolutely summable family of real numbers, so it is a summable one. Define $z(k):=\sum_{\gamma\in \Gamma} x_\gamma(k)$ for each $k\in I$. Obviously $z\in \ell_\infty(I)$ and verifies $\Vert z\Vert\leq 1+\varepsilon$. Moreover, $z\in X^*$ because every ball of $X^*$ is weak-star closed. Moreover, we have the following Claim.

\begin{claim}
$\Vert z\Vert\geq 1-\varepsilon$.
\end{claim}

\begin{proof}
Pick $\delta>0$ and $\gamma_0\in \Gamma$. As $x_{\gamma_0}$ has norm one there exists $k\in I$ such that $\vert x_{\gamma_0}(k)\vert>1-\delta$. Pick a finite subset $F_0\subseteq \Gamma$ such that given a finite subset $F_0\subseteq F$ then $\left\vert z(k)-\sum_{\gamma\in F}x_\gamma(k)\right\vert<\frac{\delta}{2}$. Now
\begin{eqnarray*}
1+\varepsilon&\geq&\vert z(k)\vert\geq \left\vert \sum_{\gamma\in F_0\cup\{\gamma_0\}} x_\gamma(k)\right\vert-\left\vert z(k)-\sum_{\gamma\in F_0\cup\{\gamma_0\}}x_\gamma(k)\right\vert\\
&>&\vert x_{\gamma_0}(k)\vert-\sum_{\gamma\in F_0\setminus\{\gamma_0\}}\vert x_\gamma(k)\vert-\frac{\delta}{2} \mathop{>}\limits^{\mbox{(\ref{sumabsolutafami})}}1-\frac{\delta}{2}-\varepsilon-\frac{\delta}{2}-\frac{\delta}{2}.
\end{eqnarray*}
As $\delta>0$ was arbitrary we conclude that $\Vert z\Vert\geq 1-\varepsilon$.
\end{proof} 
Define $u:=\frac{z}{\Vert z\Vert}$ and pick $\delta>0$. By condition (\ref{defi211}) of definition of unconditional almost squareness we can find $\gamma_0\in \Gamma$ such that
\[ \Vert u\pm x_{\gamma_0}\Vert\leq 1+\delta. \]
Pick $p\in I$ such that $\vert x_{\gamma_0}(p)\vert>1-\delta$. Pick also $F_0$ a finite subset of $\Gamma$ such that given a finite subset $F_0\subseteq F$ then $\left\vert z(p)-\sum_{\gamma\in F} x_\gamma(p)\right\vert<\frac{\delta}{2}$. Now
\begin{eqnarray*}
1+\varepsilon & \geq & \vert u(p)+x_{\gamma_0}(p)\vert=\left\vert \frac{z(p)+x_{\gamma_0}(p)}{\Vert z\Vert}+x_{\gamma_0}(p)-\frac{x_{\gamma_0}(p)}{\Vert z\Vert} \right\vert\\
&\geq & \frac{\vert z(p)+x_{\gamma_0}(p)\vert}{\Vert z\Vert}-\frac{\vert x_{\gamma_0}(p)\vert\ \vert \Vert z\Vert-1\vert}{\Vert z\Vert}>\frac{\vert z(p)+x_{\gamma_0}(p)\vert}{\Vert z\Vert}-\frac{\varepsilon}{1+\varepsilon}.
\end{eqnarray*}
Moreover
\begin{eqnarray*} \vert z(p)+x_{\gamma_0}(p)\vert & \geq & \left\vert \sum_{\gamma\in F_0\cup \{\gamma_0\}}x_\gamma(p)+x_{\gamma_0}(p)\right\vert -\left\vert z(p)-\sum_{\gamma\in F_0\cup\{\gamma_0\}}x_\gamma(p)\right\vert\\
&>&2\vert x_{\gamma_0}(p)\vert-\sum_{\gamma\in F_0\setminus\{\gamma_0\}}\vert x_\gamma(p)\vert-\frac{\delta}{2}\\
&>& 2(1-\delta)-\varepsilon -\delta-\frac{\delta}{2}>2-4\delta-\varepsilon.
\end{eqnarray*} 
Summarising one has
$$1+\varepsilon\geq \vert u(p)+x_{\gamma_0}(p)\vert>\frac{2-4\delta-\varepsilon}{1-\varepsilon}-\frac{\varepsilon}{1+\varepsilon},$$
which obviously does not hold for values of $\varepsilon$ and $\delta$ small enough.

Consequently, there is not any dual and  UASQ Banach space.
\end{proof}

Note that this provides a partial answer to the problem of whether there exists any dual ASQ Banach space posed in \cite{all} and \cite{blr2}. Above Theorem says that the answer is no if we additionally assume unconditional almost squareness. 

Now we will end the section with a stability result on unconditional almost squareness which will be used in Section \ref{secvectorvalued}. However, the proof is an straightforward adaptation of the one given for ASQ Banach spaces in \cite[Theorem 2.6]{llr}.

\begin{proposition}\label{estabiluasqespaope}
Let $X$ and $Y$ be non-zero Banach spaces and assume that $H$ is a subspace of $K(Y^*,X)$ such that $X\otimes Y\subseteq H$. If $X$ is UASQ, so does $H$.
\end{proposition}

\begin{proof}
Pick a positive $\varepsilon>0$. As $X$ is UASQ we can find a set $\{x_\gamma\}_{\gamma\in\Gamma}\subseteq S_X$ verifying the conditions of Definition \ref{definitiouasq}. Now, given $y\in S_Y$, for each $\gamma\in \Gamma$ define $T_\gamma:=x_\gamma\otimes y$, which is by assumptions a norm-one element of $H$. We will prove that $\{T_\gamma\}_{\gamma\in \Gamma}$ satisfies the conditions of Definition \ref{definitiouasq}.

On the one hand, pick $F\subseteq \Gamma$ a finite subset. Now, for every $\gamma\in F$ and every $\xi_\gamma\in\{-1,1\}$, one has
$$\left\Vert \sum_{\gamma\in F} \xi_\gamma T_\gamma \right\Vert=\left\Vert \sum_{\gamma\in F} \left(\xi_\gamma x_\gamma\right)\otimes y \right\Vert=\left\Vert \sum_{\gamma\in F} \xi_\gamma x_\gamma\right\Vert \Vert y\Vert\leq 1+\varepsilon.$$
On the other hand, pick $S_1,\ldots, S_k\in S_H$ and $\varepsilon>0$. Consider the relative compact set $K:=\bigcup\limits_{i=1}^k S_i(B_{Y^*})$. Now we can find $x_1,\ldots, x_n\in B_X$ such that $K\subseteq \bigcup\limits_{i=1}^n B\left(x_i,\frac{\delta}{2}\right)$. 

Define $E:=span\{x_i:i\in\{1,\ldots, n\}\}$, which is a finite-dimensional subspace of $X$. Now we will get the following Claim adapting the proof of \cite[Theorem 2.4]{all}.
\begin{claim}\label{claimtensor} There exists $\gamma\in \Gamma$ such that
\[ \left(1-\frac{\delta}{2}\right)\max\{||x||,|\lambda|\} \leq ||x+\lambda x_\gamma|| \leq \left(1+\frac{\delta}{2}\right)\max\{||x||,|\lambda|\} \]
for all $\lambda\in\real$ and $x\in E$. 
\end{claim}
\begin{proof} Let $\{y_i\}$ be a finite $\delta/4$-net in $S_E$ and take $\gamma\in \Gamma$ such that $||y_i\pm x_\gamma||\leq 1+\delta/4$ for every $i$. Thus, $||x\pm x_\gamma||\leq 1+\frac{\delta}{2}$ for every $x\in E$. We also have $||x\pm x_\gamma||\geq 1-\frac{\delta}{2}$.  Now \cite[Lemma 2.2]{all} does the work.
\end{proof}
Taking $\gamma$ as in the Claim, we will prove that $\Vert S_i\pm T_\gamma\Vert\leq 1+\delta$ for every $i\in\{1,\ldots, k\}$. To this aim, pick $i\in\{1,\ldots, k\}$ and $y^*\in B_{Y^*}$. From the condition on $x_1,\ldots, x_n$ we conclude the existence of $j\in\{1,\ldots, n\}$ such that $\Vert S_i(y^*)-x_j\Vert<\frac{\delta}{2}$. Now
$$\Vert (S_i\pm T_\gamma)(y^*)\Vert\leq \Vert S_i(y^*)-x_j\Vert+\Vert x_j\pm y^*(y)x_\gamma\Vert$$
$$\leq \frac{\delta}{2}+\left(1+\frac{\delta}{2} \right)\max\{\Vert x_j\Vert, \vert y^*(y)\vert\}\leq 1+\delta.$$
Now, taking supremum in $y^*\in B_{Y^*}$, we get
$$\Vert S_i\pm T_\gamma\Vert=\sup\limits_{y^*\in B_{Y^*}}\Vert (S_i\pm T_\gamma)(y^*)\Vert\leq 1+\delta.$$
As $\delta$ was arbitrary we conclude that $H$ is UASQ, so we are done.
\end{proof}

\section{Unconditional almost squareness in \texorpdfstring{$lip(M)$}{lip(M)} and \texorpdfstring{$S(M)$}{S(M)}}\label{seccasonodiscre}

\bigskip

In this Section we will provide several examples of little-Lipschitz UASQ Banach spaces. Some of them will have important consequences in the corresponding vector-valued version. 
We shall begin by exhibiting a result for proper spaces.

\begin{proposition}\label{smfindimouasq}
Let $M$ be a pointed proper space. Then $S(M)$ is either finite-dimensional or UASQ. In particular, when $M$ is compact, the same conclusion holds for $lip(M)$.
\end{proposition}

\begin{proof}
By \cite[Lemma 3.9]{dal2} it is known that $S(M)$ is $(1+\varepsilon)$-isometric to a subspace of $c_0$. Then, either $S(M)$ is finite dimensional or $S(M)$ is not reflexive. In the non-reflexive case, we get that $S(M)$ is $(1+\varepsilon)$-isometric to an $M$-embedded Banach space, so $S(M)$ is a non-reflexive $M$-embedded space. By \cite[Corollary 4.3]{all}, $S(M)$ is ASQ. As $S(M)$ is separable then it is UASQ by Corollary \ref{uasqequiasqsep}. 
\end{proof}

\begin{remark} Recall that a Banach space $X$ is said to be \textit{octahedral} if given $Y$ a finite dimensional subspace of $X$ and a positive $\varepsilon$ there is $x\in S_X$ such that $\Vert y+\lambda x\Vert\geq (1-\varepsilon)(\Vert y\Vert+\vert \lambda\vert)$ for every $y\in Y$ and every $\lambda\in\mathbb R$. In \cite{blr} the question of whether $\mathcal F(M,X)$ has an octahedral norm is analysed under the strong assumption of having the pair $(M,X^*)$ the so-called the contraction-extension property. Above Proposition gives us another criterion about octahedrality in $\mathcal F(M,X)$ which avoids dealing with the contraction-extension property. Indeed, consider $M$ be a proper pointed metric space and $X$ be a Banach space. Assume that $\mathcal F(M)=S(M)^*$ and that $\mathcal F(M)$ has the approximation property (see examples in \cite{dal2}). Then Proposition \ref{smfindimouasq} implies that $\mathcal F(M,X^*)=\mathcal F(M)\pten X^*$ has an octahedral norm by \cite[Corollary 2.9]{llr}.
\end{remark}

In view of Proposition \ref{smfindimouasq} it is natural to wonder whether $lip(M)$ or $S(M)$ is an UASQ Banach space. Next Proposition yields positive partial answers.

\begin{proposition}\label{puntosaclema}
Let $M$ be a pointed metric space and let $W$ be a closed subspace of $lip(M)$. Assume that there exist sequences $\{f_n\}\subset S_W$, $\{x_n\}\subset M$ and $\{r_n\}$ of positive numbers such that $r_n\rightarrow 0$, $f_n(x_n)=0$ and $f_n(t)=0$ for each $t\in M\setminus B(x_n,r_n)$. Then $W$ is an UASQ Banach space.
\end{proposition}

\begin{proof}
We will prove that $W$ verifies the assumptions of Proposition \ref{condisufiuasq} for the sequence $\{f_n\}$. To this aim pick $\varepsilon>0$ and $g\in S_{W}$. We can find, by little Lipschitz condition, a positive $\delta>0$ such that 
\begin{equation}\label{litlipconditheo22}
0<d(x,y)<2\delta \Rightarrow \frac{\vert g(x)-g(y)\vert}{d(x,y)}<\varepsilon.
\end{equation}
Pick $m\in\mathbb N$ such that $n\geq m$ implies $r_n<\delta$ and $\frac{r_n}{\delta-r_n}<\varepsilon$.

Consider $n\geq m$ and let us estimate $\Vert g\pm f_n\Vert$. To this aim pick $x,y\in M, x\neq y$. One has
$$\frac{\vert g(x)\pm f_n(x)-(g(y)\pm f_n(y)) \vert}{d(x,y)}\leq \underbrace{\frac{\vert g(x)-g(y)\vert}{d(x,y)}}_A+\underbrace{\frac{\vert f_n(x)-f_n(y)\vert}{d(x,y)}}_B:=C$$
Now we shall distinguish cases, depending on the position of $x$ and $y$.
\begin{enumerate}
\item $x,y\notin B(x_n,\delta)$. In this case $B=0$ and, consequently, $C\leq 1$.
\item $x,y\in B(x_n,\delta)$. In this case $A\leq \varepsilon$ because of (\ref{litlipconditheo22}), so $C\leq 1+\varepsilon$.
\item $x\in B(x_n,\delta), y\notin B(x_n,\delta)$. In this case $f_n(y)=0$. We can even distinguish two more cases here:
\begin{enumerate}
\item $x\notin B(x_n,r_n)$. Here $f_n(x)=f_n(y)=0$, so $B=0$ and thus $C\leq 1$.
\item $x\in B(x_n,r_n)$. In this case, by triangle inequality, $d(x,y)>\delta-r_n$. Consequently 
$$B\leq \frac{\vert f_n(x)\vert}{d(x,y)}\leq \frac{r_n}{\delta-r_n}<\varepsilon,$$
so $C\leq 1+\varepsilon$ in this case.
\end{enumerate}
\end{enumerate}

Therefore, taking supremum in $x\neq y$ we get that $\Vert g \pm f_n\Vert\leq 1+\varepsilon$, as desired. Thus Proposition \ref{condisufiuasq} implies that $W$ is UASQ. 
\end{proof}

Now we will provide examples of metric spaces $M$ for which above Proposition applies under some topological assumptions on $M$.

\begin{proposition}\label{puntosactotdis}
Let $M$ be a locally compact pointed metric space. Assume that $0\in M'$. Assume also that either $M$ is totally disconnected or $S(M)$ separates the points uniformly (i.e. there exists a constant $c\geq 1$ such that for every $x, y\in M$ there is $f\in S(M)$ satisfying $||f||\leq c$ and $|f(x)-f(y)|=d(x, y)$). Then both $lip(M)$ and  $S(M)$ are UASQ Banach spaces.
\end{proposition}

\begin{proof}
It is enough to prove, in view of above Proposition, that for every $\delta>0$ there exists $f\in S_{lip(M)}$ such that $f(t)=0$ for every $t\in M\setminus B(0,\delta)$ (note that, in this case, $f$ trivially belongs to $S(M)$).

First of all, assume that $M$ is totally disconnected. In  that case, $M$ has a basis formed by clopen sets because $M$ is a totally disconnected and locally compact metric space \cite[Proposition 2.6.8]{coo}. Pick a positive number $\delta>0$.

Note that, by assumptions, we can consider $K$ to be a neighbourhood of zero which is both clopen and compact (closed subsets of compact topological spaces are compact) and which is contained in $B(0,\delta)$. As $0\in M'$ we can find $t\in K\setminus \{0\}$. As before we can find $C$ to be a clopen neighbourhood of $0$ not containing $t$. Now define the following function
\[ h(x):=\begin{cases}
1 & \mbox{if } x\in K\setminus C\\
0 & \mbox{otherwise}
\end{cases}. \]
Now we have that $h(0)=0$ and $h(x)=0$ for every $x\in M\setminus K$. Let us prove that $h\in lip(M)$. To this aim define $A_1:=C, A_2:= K\setminus C, A_3:=M\setminus K$. Then $A_1$ and $A_2$ are compact sets (are closed subsets of a compact one) and $A_3$ is closed. Consequently, $\alpha:=\min\limits_{i\neq j}dist(A_i,A_j)>0$. Therefore, given $x,y\in M, x\neq y$, one has
\[\frac{\vert h(x)-h(y)\vert}{d(x,y)}=\begin{cases}
\frac{1}{d(x,y)}\leq \frac{1}{\alpha} & \mbox{if } x\in A_2, y\notin A_2\\
0 & \mbox{otherwise}
\end{cases},\]
which proves that $h\in Lip(M)$. Moreover, given $0<d(x,y)<\alpha$ then there exists $i$ such that both $x,y\in A_i$. Consequently
$$\frac{\vert h(x)-h(y)\vert}{d(x,y)}=0.$$
This proves that $h\in lip(M)$. So $f:=\frac{h}{\Vert h\Vert}$ yields the desired function.

Now assume that $S(M)$ separates the points of $M$ uniformly. In that case, we can find $\varepsilon<\delta$ such that $\overline{B}(0,\varepsilon)$ is compact. Consider $0<\eta<\varepsilon$ and pick $t\in B(0,\eta)\setminus\{0\}$. Define the following function on $N:=\{0,t\}\cup \overline{B}(0,\varepsilon)\setminus B(0,\eta)$:
\[ h(x):= \begin{cases}
     1 & \mbox{if }x=t \\
     0 & \mbox{otherwise}
\end{cases} \]
It is quite obvious that $h$ is a little Lipschitz function on $N$, which is a closed subset of the compact one given by $\overline{B}(0,\varepsilon)$. By \cite[Theorem 3.2.6]{wea} then there exists $g\in lip(\overline{B}(0,\varepsilon))$ an extension of $h$. Now $g$ is a non-zero function which vanishes out of $B(0,\eta)$. Consequently, if we define
\[ f(t):=\begin{cases}
g(t) & \mbox{if } t\in \overline{B}(0,\varepsilon)\\
0 & \mbox{otherwise,}
\end{cases}.\]
it is quite clear that $f\in lip(M)$. Furthermore, $f\neq 0$ as $f(x)\neq 0$. So $\frac{f}{\Vert f\Vert}$ is the desired function.
\end{proof}

\begin{remark}\label{remarkpuntosac} 
According to \cite{dal2}, a metric space $M$ is said to be \textit{ultrametric} if for every $x,y,z\in M$ it follows $d(x,z)\leq \max\{d(x,y),d(y,z)\}$. Note that construction involving the proof of Proposition \ref{puntosactotdis} still works for the class of metric spaces having a cluster point which are bi-Lipschitz equivalent to an ultrametric space. In fact, given such a metric space $M$, there exists a positive constant $C>0$ such that for each $x\in M$ and $R>0$ we can find a closed set $A$ such that $A\subseteq B(x,R)$, $B(x,C^{-1}R)\subseteq A$ and $dist(A,M\setminus A)\geq C^{-1}R$ \cite[Proposition 15.7]{dase}. We thank to Michal Doucha for addressing us to this Remark.
\end{remark}

First part of Proposition \ref{puntosactotdis} can be see as a part of a more general result.

\begin{theorem}\label{discnouniuasq}

Let $M$ be a locally compact and totally disconnected pointed metric space. If $M$ is not uniformly discrete then both $lip(M)$ and $S(M)$ are UASQ.

\end{theorem}

\begin{proof}

On the one hand, assume that $M'\neq \emptyset$. In that case we can choose freely the origin $0$ to be a cluster point. In that case, Propositions \ref{puntosactotdis} finishes the proof.

On the other hand, assume that $M'=\emptyset$. As $M$ is not uniformly discrete then there exists a pair of sequences $\{x_n\},\{y_n\}$ in $M$ such that $0<d(x_n,y_n)$ and the sequence of last distances converges to zero. We can assume, as $0\notin M'$, that $x_n\neq 0$ and $y_n\neq 0$ for each $n\in \mathbb N$.

For every $n\in\mathbb N$ define $\alpha_n:=\inf\{d(x_n,y):y\neq x_n\}$. Now $\{\alpha_n\}$ is a sequence of positive numbers (because $M$ is a discrete metric space) which is null (because $\alpha_n\leq d(x_n,y_n$ for every $n\in \mathbb N$).

For each $n\in\mathbb N$, define $f_n:=\alpha_n\chi_{\{x_n\}}$. First of all, we will prove that $f\in Lip(M)$. To this aim, pick $x,y\in M, x\neq y$. Now

\begin{equation}\label{dicocasodiscrenounifdiscre}
\frac{\vert f_n(x)-f_n(y)\vert}{d(x,y)}\neq 0\mbox{ if, and only if, } x=x_n\mbox{ or } y=x_n.\end{equation}
When $x=x_n$ then we have the estimate
$$\frac{\vert f(x_n)\vert}{d(x_n,y)}=\frac{\alpha_n}{d(x_n,y)}\leq 1.$$
Consequently, $f_n\in B_{Lip(M)}$ for each $n\in \mathbb N$. It is straightforward from the definition of infimum that previous functions are actually norm-one ones. Now we shall prove the little-Lipschitz condition. This is an obvious consequence from the fact that given $n\in\mathbb N$ and $x,y\in M$ such that $0<d(x,y)<\alpha_n$ then both $x$ and $y$ are different to $x_n$. Consequently 
$$\frac{\vert f_n(x)-f_n(y)\vert}{d(x,y)}=0.$$
Therefore $\{f_n\}$ is a sequence in $S_{lip(M)}$. Indeed, $\{f_n\}$ is still contained in $S(M)$, because given $x,y\in M, x\neq y$, we get
$$\frac{\vert f(x)-f(y)\vert}{d(x,y)}\mathop{\leq}\limits^{\mbox{
(\ref{dicocasodiscrenounifdiscre})}}\frac{\alpha_n}{d(x_n,y)}\stackrel{d(y,0)\rightarrow\infty}{\xrightarrow{\hspace*{1cm}}}0.$$
So $\{f_n\}$ is a sequence in the unit sphere of $S(M)$. Notice that $f_n(y_n)=0$ and $f_n$ vanishes out of $B(y_n, 2d(x_n,y_n))$. Therefore, the sequence $\{f_n\}$ satisfies the assumptions of Proposition \ref{puntosaclema}, which proves that both $lip(M)$ and $S(M)$ are UASQ. 
\end{proof}

As a consequence we get the following Corollary.

\begin{corollary}
Let $M$ be a locally compact and totally disconnected metric space with a designated origin. Assume that $M$ is not uniformly discrete and let $X$ be either $lip(M)$ or $S(M)$. Then:
\begin{enumerate}
\item $X$ contains $1+\varepsilon$-isometric copies of $c_0$ \cite[Lemma 2.6]{all}.
\item $X$ can not be isometric to any dual Banach space.
\item $X^*$ has an octahedral norm and, consequently, contains an isomorphic copy of $\ell_1$ \cite[Proposition 2.5]{all}.
\end{enumerate}
\end{corollary}

\section{The uniformly discrete case.}\label{secunifdisc}

\bigskip

Throughout this section $M$ will denote an uniformly discrete metric space, that is, $\inf_{x\neq y} d(x,y)>0$. We shall begin by analysing the bounded case, for which we have the equalities $S(M)=lip(M)=Lip(M)$. Hence, previous spaces can not be UASQ as being dual spaces. However, we will actually prove that $Lip(M)$ even fails to be an ASQ space.

\begin{proposition}\label{unifdiscacotnotasq}

Let $M$ be an uniformly discrete and bounded metric space with a designated origin $0$. Then $Lip(M)$ is not ASQ.

\end{proposition}

\begin{proof}
Define $D:=\inf\limits_{x\neq y} d(x,y)$ and $R(M):=\sup\limits_{x\in M}d(x,0)$. Consider also $f(x):=d(x,0)$, which is a norm-one Lipschitz function and $\varepsilon>0$ such that $\frac{2R(M)}{D}\varepsilon<\frac{1}{2}$. We will see that $f$ and $\varepsilon$ fail the definition of ASQ. For this, pick $g\in S_{Lip(M)}$ such that $\Vert f\pm g\Vert\leq 1+\varepsilon$ and let us prove that $\Vert g\Vert<\frac{1}{2}$. To this aim, pick $x\in M, x\neq 0$. Then
$$1+\varepsilon\geq \frac{\vert d(x,0)\pm g(x)\vert}{d(x,0)}.$$
Now taking a suitable choice of sign we get
$$1+\varepsilon\geq \frac{d(x,0)+\vert g(x)\vert}{d(x,0)}=1+\frac{\vert g(x)\vert}{d(x,0)}\Rightarrow \vert g(x)\vert\leq \varepsilon d(x,0).$$
Finally, given $x,y\in M, x\neq y$, previous inequality yields
$$\frac{\vert g(x)-g(y)\vert}{d(x,y)}\leq \varepsilon \frac{d(x,0)+d(y,0)}{d(x,y)}\leq \frac{2R(M)}{D}\varepsilon.$$
Consequently, taking supremum in the above estimate we get that $\Vert g\Vert\leq \frac{2R(M)}{D}\varepsilon<\frac{1}{2}$, as desired.
\end{proof}

Now we will explore the unbounded case, in which $S(M)\neq lip(M)=Lip(M)$. Even though we will not get a general result of non almost squareness on $Lip(M)$, we will at least exhibit some examples of unbounded metric spaces in which $Lip(M)$ is not ASQ. To this aim, given $M$ a metric space, we will consider
$$N(M):=\left\{\frac{\delta_x-\delta_y}{d(x,y)}: x\neq y \right\}.$$
It is clear that $N(M)$ is a norming subset for $Lip(M)$. Now we shall exhibit a characterisation of the fact that $0$ belongs to the weak closure of $N(M)$, which will be applied to study almost squareness on $Lip(M)$. We thank to G. Lancien for addressing us to the following proof. 

\begin{lemma}\label{cara0cierredebil}

Let $M$ be a metric space with a designated origin $0$. Then, $0\notin \overline{N(M)}^w$ if, and only if, $M$ can be bi-Lipschitz embedded into a $\ell_\infty^k$, for some natural number $k$.

\end{lemma}

\begin{proof}
First of all, assume that there exists a natural number $k$ and $f\colon M\rightarrow \ell_\infty^k$ a bi-Lipschitz embedding. Then $f:=(f_1,\ldots, f_k)$ for suitable $f_i\in Lip(M)$. As $f$ is a bi-Lipschitz embedding we can find $\varepsilon>0$ such that 
$$\forall x,y\in M, x\neq y \mbox{ one has } \frac{\Vert f(x)-f(y)\Vert}{d(x,y)}\geq \varepsilon.$$
In view of the norm on $\ell_\infty^k$, previous condition is equivalent to the following one:
$$\forall x,y\in M, x\neq y\ \exists\ i\in\{1,\ldots, k\} \mbox{ such that } \frac{\vert f_i(x)-f_i(y)\vert}{d(x,y)}\geq \varepsilon.$$
This means that the weak neighbourhood of $0$ in $\mathcal F(M)$ given by $V:=\{\gamma\in \mathcal F(M): \vert f_i(\gamma)\vert<\varepsilon\ \forall i\in\{1,\ldots, k\} \}$ satisfies that $V\cap N(M)=\emptyset$, so $0\notin \overline{N(M)}^w$.

Conversely, assume that $0\notin \overline{N(M)}^w$. This means that there exists a weak neighbourhood of $0$ in $\mathcal F(M)$, say  $V:=\{\gamma\in \mathcal F(M): \vert f_i(\gamma)\vert<\varepsilon\ \forall i\in\{1,\ldots, k\} \}$ such that $V\cap N(M)=\emptyset$, where $\varepsilon>0$ and $f_1,\ldots, f_k\in Lip(M)$. Define $f=(f_1,\ldots, f_k)$. Then $f\colon M\longrightarrow \ell_\infty^k$ is a Lipschitz map. We will check that $f$ is a bi-Lipschitz embedding. To this aim pick $x,y\in M, x\neq y$. By assumption $\frac{\delta_x-\delta_y}{d(x,y)}\notin V$, so there exists $i\in\{1,\ldots,  k\}$ such that $\frac{\vert f_i(x)-f_i(y)\vert}{d(x,y)}\geq \varepsilon$. Consequently $\frac{\Vert f(x)-f(y)\Vert}{d(x,y)}\geq \varepsilon$. As $x,y\in M$ were arbitrary we conclude that $f$ is a bi-Lipschitz embedding into $\ell_\infty^k$, so we are done.
\end{proof}

From here we will get some examples of metric spaces $M$ such that $Lip(M)$ can not be an ASQ Banach space.

\begin{proposition}\label{0nocierreLipnoasq}
Let $M$ be a metric space which is bi-Lipschitz embeddable into any $\ell_\infty^k$. Then $Lip(M)$ is not an ASQ Banach space.
\end{proposition}

\begin{proof}
Assume, by contradiction, that $Lip(M)$ is ASQ, and let us prove that $0\in \overline{N(M)}^w$. To this aim pick $f_1,\ldots, f_k\in S_{Lip(M)}$ and $\varepsilon>0$. Let us prove that $V:=\{\gamma\in \mathcal F(M): \vert f_i(\gamma)\vert<\varepsilon\ \forall\ i\in \{1,\ldots, k\} \}$ intersects $N(M)$. As we are assuming that $Lip(M)$ is ASQ we can find $g\in S_{Lip(M)}$ such that
$$\Vert f_i\pm g\Vert\leq 1+\frac{\varepsilon}{2}\ \forall i\in \{1,\ldots, k\}.$$
Since $g$ is a norm-one function we can find $x,y\in M, x\neq y$ such that $\frac{\vert g(x)-g(y)\vert}{d(x,y)}>1-\frac{\varepsilon}{2}$. Pick an arbitrary $i\in \{1,\ldots, k\}$. Then 
$$1+\frac{\varepsilon}{2}\geq \frac{\vert f_i(x)-f_i(y)\pm (g(x)-g(y))\vert}{d(x,y)}.$$
Taking a suitable choice of sign we deduce that
$$1+\frac{\varepsilon}{2}\geq \frac{\vert f_i(x)-f_i(y)\vert}{d(x,y)}+\frac{\vert g(x)-g(y)\vert}{d(x,y)}>\frac{\vert f_i(x)-f_i(y)\vert}{d(x,y)}+1-\frac{\varepsilon}{2}.$$
Consequently $\frac{\vert f_i(x)-f_i(y)\vert}{d(x,y)}<\varepsilon$. As $i$ was arbitrary we conclude that $\frac{\delta_x-\delta_y}{d(x,y)}\in V$. From here we get that $0\in \overline{N(M)}^w$, which contradicts Proposition \ref{cara0cierredebil}. Thus, $Lip(M)$ can not be an ASQ Banach space, as desired.
\end{proof}

In spite of Proposition \ref{0nocierreLipnoasq}, we get even unconditional almost squareness whenever we restrict to a smaller subspace.

\begin{proposition}\label{unidiscsasq}
Let $M$ be an unbounded and uniformly discrete metric space. Then $S(M)$ is UASQ.

\end{proposition}

\begin{proof}
Pick $\{x_n\}$ to be a sequence in $M$ such that $\{d(x_n,0)\}\rightarrow \infty$. Define, for each $n\in\mathbb N$, $\alpha_n:=\sup\{R>0: B(x_n,R)=\{x_n\} \}$ and define $f_n:=\alpha_n \chi_{\{x_n\}}$, which is a norm-one Lipschitz function. Indeed, it is straightforward to prove that $f_n\in S(M)$. We will prove that $\{f_n\}$ verifies the assumptions of Proposition \ref{condisufiuasq}. To this aim pick $g\in S_{S(M)}$ and $\varepsilon>0$. As $g\in S(M)$ we can find $R>0$ such that
$$ x\mbox{ or }y\notin B(0,R)\Rightarrow \frac{\vert g(x)-g(y)\vert}{d(x,y)}<\varepsilon.$$
Pick $m\in\mathbb N$ verifying that $n\geq m$ implies $d(x_n,0)>R$.

Pick $n\geq m$ and let us estimate $\Vert g\pm f_n\Vert$. To this aim consider $x,y\in M, x\neq y$. Then
$$\frac{\vert g(x)\pm f_n(x)-(g(y)\pm f_n(y)) \vert}{d(x,y)}\leq \underbrace{\frac{\vert g(x)-g(y)\vert}{d(x,y)}}_A+\underbrace{\frac{\vert f_n(x)-f_n(y)\vert}{d(x,y)}}_B:=C$$
We shall discuss by cases:
\begin{enumerate}
\item If $x\notin B(0,R)$ or $y\notin B(0,R)$ then $A<\varepsilon$ and so $C\leq 1+\varepsilon$.
\item If both $x,y\in B(0,R)$ then by definition of $f_n$ we get $f_n(x)=f_n(y)=0$ and so $C\leq 1$.
\end{enumerate}
In any case we get $\Vert g\pm f_n\Vert\leq 1+\varepsilon$, so we are done.\end{proof}

\section{Vector valued case}\label{secvectorvalued}

\bigskip

In \cite[Proposition 3.7]{jsv} it is proved that if $M$ is compact then $lip(M,X)$ is  linearly isometrically isomorphic to the space of compact operators from $X^*$ to $lip(M)$ which are continuous for the bounded weak-star topology. We will extend that result to the case of proper metric spaces.

The following result, which is a slight generalisation of Theorem 3.2 in \cite{johnson}, give us a criterion for compactness in $S(M)$. 

\begin{proposition}\label{relcomp}

Let $M$ be a proper pointed metric space and $\mathcal F$ be a subset of $S(M)$. Then the following are equivalent:
\begin{enumerate}
\item[i)] $\mathcal F$ is relatively compact in $S(M)$.
\item[ii)] $\mathcal F$ is bounded and satisfies the $S(M)$-condition uniformly, that is, for each $\varepsilon>0$ there exist $\delta>0$ and $r>0$ such that 
\begin{align*}
\sup_{\stackrel{x \text{ or } y\notin B(0,r)}{x\neq y}} \frac{| f(x)-f(y)|}{d(x,y)}<\varepsilon, &
\sup_{
0<d(x,y)<\delta} \frac{| f(x)-f(y)|}{d(x,y)}<\varepsilon
\end{align*}
for every $f\in\mathcal F$. 
\end{enumerate}

\end{proposition}

\begin{proof} Notice that $M\times M\setminus \Delta$ is a locally compact space, where $\Delta=\{(x,x):x\in M\}$. Let $K=(M\times M\setminus \Delta) \cup\{\infty\}$ be its one-point compactification. Given $f\in S(M)$, consider $\tilde{f}:K\to\real$ defined by $\tilde{f}(x,y)=\frac{f(x)-f(y)}{d(x,y)}$ and $\tilde{f}(\infty)=0$. Clearly $\tilde{f}$ is continuous at each $(x,y)\in M\times M\setminus \Delta$. Moreover, given $\varepsilon>0$ there exist $r>0$ and $\delta>0$ such that $\tilde{f}(x,y)<\varepsilon$ whenever $(x,y)$ belongs to the complementary of the compact set $\overline{B}(0,r)\times\overline{B}(0,r)\cap\{(s,t): d(s,t)\geq\delta\}$. So $\tilde{f}$ is continuous at $\infty$. Thus $f\mapsto \tilde{f}$ defines a linear isometry from $S(M)$ into $C(K)$. By the Ascoli-Arzel\`a theorem, a subset $\mathcal F$ of $S(M)$ is relatively compact if, and only if, $\tilde{\mathcal F}$ is equicontinuous and bounded in $C(K)$. Clearly $\mathcal F$ is bounded in $S(M)$ if, and only if, $\tilde{\mathcal F}$ is bounded in $C(K)$. Moreover,
\[ |\tilde{f}(x,y)-\tilde{f}(s,t)|\leq ||f||_L \left\Vert \frac{\delta_x-\delta_y}{d(x,y)}-\frac{\delta_s-\delta_t}{d(s,t)}\right\Vert\]
for every $f\in \mathcal F$. Thus, if $\mathcal F$ is bounded then $\tilde{\mathcal F}$ is equicontinuous at each $(x,y)\in M\times M\setminus \Delta$. Therefore it suffices to show that $\mathcal F$ satisfies $S(M)$-condition uniformly if an only if $\tilde{\mathcal F}$ is equicontinuous at $\infty$. That follows from the fact that the family $\mathcal U = \{U_{r,\delta}\}_{r>0,\delta>0}$ is a neighbourhood basis of $\infty$ in $K$, where \[U_{r,\delta} = K\setminus(\overline{B}(0,r)\times\overline{B}(0,r)\cap\{(s,t): d(s,t)\geq\delta\}),\]
which finishes the proof.\end{proof}

Relative compactness condition given by the above Proposition will be the key for proving the following Theorem.

\begin{theorem}\label{scomp} 

Let $M$ be a proper pointed metric space. Then $S(M,X)$ is linearly isometrically isomorphic to $K_{w^*,w}(X^*,S(M))$.

\end{theorem}

\begin{proof}  It is shown in \cite{jsv} that $f\mapsto f^t$ defines a linear isometry from $Lip(M,X)$ onto $L_{w^*,w^*}(X^*, Lip(M)))$, where $f^t(x^*)=x^*\circ f$. Let $f\in S(M,X)$. Notice that for any $x\neq y\in M$ and $x^*\in X^*$ we have
\[ \frac{|x^*\circ f(x)-x^*\circ f(y)|}{d(x,y)} \leq ||x^*|| \frac{\Vert f(x)-f(y)\Vert}{d(x,y)} \]
thus $f^t(x^*)\in S(M)$. Moreover, previous inequality proves that the functions in $f^t(B_{X^*})$ satisfy the $S(M)$-condition uniformly. By Proposition \ref{relcomp} we get that $f^t(B_{X^*})$ is a relatively compact subset of $S(M)$ and thus $f^t\in K(X^*, S(M))\cap L_{w^*,w^*}(X^*, Lip(M)))$. The set $\overline{f^t(B_{X^*})}$ is norm-compact and thus every coarser Hausdorff topology agrees on it with the norm topology. In particular, the weak topology of $S(M)$ agrees on $f^t(B_{X^*})$ with the inherited weak-star topology of $Lip(M)$. Thus $f^t|_{ B_{X^*}}\colon B_{X^*} \to S(M)$ is $w^*-w$ continuous. By \cite[Proposition 3.1]{Kim} we have that $f^t\in K_{w^*,w}(X^*,S(M))$.

It only remains to prove that the isometry is onto. For this take $T\in K_{w^*,w}(X^*,S(M))$. We claim that $T$ is $w^*-w^*$ continuous from $X^*$ to $Lip(M)$. Indeed, assume that $\{x_\alpha^*\}$ is a net in $X^*$ weak-star convergent to some $x^*\in X^*$. As every $\gamma\in \mathcal F(M)$ is also an element in $S(M)^*$ then $\langle\gamma, Tx^*_\alpha\rangle$ converges to $\langle\gamma, Tx^*\rangle$. Thus, $T\in L_{w^*,w^*}(X^*, Lip(M)))$. By the isometry described above, there exists $f\in Lip(M,X)$ such that $T=f^t$. Since $f^t(B_{X^*})$ is relatively compact, given $\varepsilon>0$ there exist $r>0$ and $\delta>0$ such that 
\[ \sup_{0<d(x,y)<\delta} \frac{| x^*\circ f(x)-x^*\circ f(y)|}{d(x,y)} < \varepsilon, \ \sup_{\stackrel{x \text{ or } y\notin B(0,r)}{x\neq y}}\frac{| x^*\circ f(x)-x^*\circ f(y)|}{d(x,y)} < \varepsilon\]
for each $x^*\in B_{X^*}$. Taking supremum with $x^*\in B_{X^*}$, we get that 
\[ \sup_{0<d(x,y)<\delta} \frac{\Vert f(x)-f(y)\Vert}{d(x,y)} < \varepsilon, \ \sup_{\stackrel{x \text{ or } y\notin B(0,r)}{x\neq y}}\frac{\Vert f(x)-f(y)\Vert}{d(x,y)}, < \varepsilon\]
so $f\in S(M, X)$. Consequently, the isometry is onto and we are done.
\end{proof}

Notice that $X\otimes S(M)$ is a subspace of $K_{w^*,w}(X^*,S(M))$. By Propositions \ref{estabiluasqespaope} and \ref{smfindimouasq} we get the desired result on non duality in vector-valued case. 

\begin{theorem}\label{UASQsmvectorvaluado}

Let $M$ be a proper pointed metric space and let $X$ be a Banach space. If $S(M)$ is infinite dimensional, then $S(M,X)$ is unconditionally ASQ. In particular, it is not a dual Banach space. In particular, the same results holds for $lip(M,X)$ when $M$ is compact.

\end{theorem}

\section{Some remarks and open questions}\label{secopenproblems}

It is shown in \cite[Theorem 2.3]{blr2} that, given a Banach space $X$ which contains an isomorphic copy of $c_0$, there exists an equivalent ASQ norm in $X$. Following the construction of such a norm, it can be proved that $X$ is actually UASQ under this norm, taking as sequence the standard basis $\{e_n\}\subset c_0\subset X$. Thus, a Banach space admits an equivalent renorming to be UASQ if and only if it contains an isomorphic copy of $c_0$. So both ASQ and UASQ are equivalent properties under an isomorphic point of view. Moreover, Corollary \ref{uasqequiasqsep} shows that previous properties are isometrically equivalent in the separable case. Consequently, we pose the following open Question. 

\begin{question}\label{UASQASQequivalent}
Are UASQ and ASQ equivalent, at least, for dual Banach spaces?
\end{question}

A positive answer to above Question would imply that there is not any dual ASQ Banach space, which would solve an open problem posed in \cite{all} as well as in \cite{blr2}.

The results of Section \ref{secunifdisc} suggest the following question. 

\begin{question}
Is there any (unbounded) uniformly discrete metric space $M$ such that $Lip(M)$ is ASQ?
\end{question}

If there existed such a metric space $M$, it would solve by the positive the problem of the existence of dual ASQ Banach spaces and, consequently, would solve by the negative Question \ref{UASQASQequivalent}.

Finally, note that Theorem \ref{UASQsmvectorvaluado} given a criterion for non duality in an isometrical way. So, we can go further and wonder:

\begin{question}
Let $M$ be a pointed proper metric space such that $S(M)$ is infinite dimensional and let $X$ be a non-separable Banach space. Can be $S(M,X)$ be isomorphic to any dual Banach space?
\end{question}

\textbf{Acknowledgements}: This work was done when both authors visited the Laboratoire de Math\'ematiques de Besan\c{c}on, for which the first author was supported by a grant from Programa de Contratos Predoctorales (FPU) de la Universidad de Murcia and the second author was supported by a grant from Vicerrectorado de Internacionalizaci\'on y Vicerrectorado de Investigaci\'on y Transferencia de la Universidad de Granada, Escuela Internacional de Posgrado de la Universidad de Granada y el Campus de Excelencia Internacional (CEI) BioTic.

The authors are deeply grateful to the Laboratoire de Math\'ematiques de Besan\c{c}on for their hospitality during the stay. They also thank to M. Doucha, G. Lancien, G. L\'opez-P\'erez and  A. Proch\'azka for useful conversations which helped to improve the paper.

\end{document}